\documentclass[a4paper,leqno]{amsart}

\usepackage{amsmath,amssymb,amsthm}
\usepackage{hyperref}
\usepackage{graphics}
\usepackage{graphicx}

\newcommand{\R}{\mathbb{R}}
\newcommand{\LL}{L}

\newcommand{\mo}[1]{|#1|}

\theoremstyle{plain}
\newtheorem{theo}{Theorem}[section]
\newtheorem{lemma}[theo]{Lemma}

\newtheorem{prop}[theo]{Proposition}
\newtheorem{rem}[theo]{Remark}

\theoremstyle{definition}

\theoremstyle{remark}

\newtheorem*{rema*}{Remark}

\def\bu{{\bf{u}}}
\def\bv{{\bf{v}}}
\def\ue{u^e}
\def\ve{v^e}
\def\cn{\mathbb{C}}

\begin{document}

\title{Dispersion for the Schr\"odinger equation on networks}  

\author[V. Banica]{Valeria Banica}
\address[V. Banica]{D\'epartement de Math\'ematiques\\ Universit\'e
  d'Evry\\ Bd. F.~Mitterrand\\ 91025 Evry\\ France} 
\email{Valeria.Banica@univ-evry.fr}
\author[L. I. Ignat]{Liviu I. Ignat}
\address[L. I. Ignat]{Institute of Mathematics ``Simion Stoilow'' of the Romanian Academy\\21 Calea Grivitei Street \\010702 Bucharest \\ Romania 
\hfill\break\indent \and
\hfill\break\indent 
 BCAM - Basque Center for Applied Mathematics,\\
Bizkaia Technology Park, Building 500 Derio, Basque Country, Spain.}

\email{liviu.ignat@gmail.com}
\thanks{V.B. is partially
  supported by the ANR project ``R.A.S''. L. I. Ignat is partially supported by PN II-RU-TE 4/2010 and PCCE-55/2008
of CNCSIS--UEFISCSU Romania.
   Both authors are supported by L.E.A. project. Parts of this paper have been developed during the second author's visit to Institute Henri Poincar\'e, Paris.} 
\begin{abstract} In this paper 
we consider the Schr\"odinger equation on a network formed by a tree with the last generation of edges formed by infinite strips. We give an explicit description of the solution of the linear Schr\"odinger equation with constant coefficients. This allows us to prove dispersive estimates, which in turn are useful for solving the nonlinear Schr\"odinger equation. The proof extends also to the laminar case of positive step-function coefficients having a finite number of discontinuities.
\end{abstract}

\maketitle

\section{Introduction}
Let us first consider the linear Schr\"{o}dinger equation (LSE) on $\mathbb R$:
\begin{equation}\label{sch1}
\left\{\begin{array}{l}
  iu_t+ u_{xx} = 0,  \,x\in \R, t\in\R,\\
  u(0,x)  =u_0(x), \,x\in \R.
    \end{array}\right.
\end{equation}
The
linear semigroup $e^{it\Delta}$ has two important properties, that can be easily seen via the Fourier transform. First, the
conservation of the $\LL^2$-norm:
\begin{equation}\label{energy}
\|e^{it\Delta}u_0\|_{\LL^2(\R)}=\|u_0\|_{\LL^2(\R)}
\end{equation}
 and a dispersive estimate of the form:
\begin{equation}\label{linfty}
\|e^{it\Delta}u_0\|_{L^\infty(\R)}\leq \frac C{\sqrt{|t|}}\|u_0\|_{\LL^1(\R)}, \
\ t\neq 0.
\end{equation}
From these two inequalities, by using the classical $TT^*$ argument, space-time estimates follow,  known as Strichartz estimates (\cite{Str},\cite{MR801582}):
\begin{equation}\label{dsch0001}
    \|e^{it\Delta}u_0\|_{\LL^q_t(\R,\,\LL^r_x(\R))}\leq C\|u_0\|_{\LL^2(\R)},
\end{equation}
 where $(q,r)$ are so-called admissible pairs:
 \begin{equation}\label{adm}
\frac 2 {q}+\frac 1r=\frac 12, \quad 2\leq q,r\leq \infty.
\end{equation}
 These dispersive estimates have been successfully applied to obtain well-posedness results  for the nonlinear Schr\"odinger equation (see  \cite{1055.35003}, \cite{MR2233925} and the reference therein).

 In this article we prove the dispersion inequality for the linear Schr\"odinger operator defined on a tree (bounded, connected graph without closed paths) with the external edges infinite.
 We assume that the tree does not contain vertices of multiplicity two, since they are irrelevant for our model. {{Let us notice that in this context we cannot use Fourier analysis as done on $\mathbb R$ for getting the dispersion inequality.}}
 
 The  presentation of the Laplace operator will be given in full details in the next section. Let us just say here that the Laplacian operator $\Delta_\Gamma$ acts as the usual Laplacian on $\mathbb R$ on each edge, and that at vertices the Kirchhoff conditions must be fulfilled: continuity condition for the functions on the graph and transmission condition at the level of their first derivative. So our analysis will be a 1-D ramified analysis. More general coupling conditions are discussed in Section \ref{open}.   

In \cite{Liviu} the second author proved the same result in the case of regular trees. This means some
restrictions on the shape of the trees: all the vertices of the same generation have the same number of descendants  and all the edges of the same generation are of the same length.
These restrictions allow to define some average functions on the edges of the same generation and  to analyze some  1-D laminar Schr\"odinger equation (depending on the shape of the tree), where dispersion estimates were available from the first author's paper \cite{Valeria}. The strategy used in \cite{Liviu} cannot be applied in the case of a general tree and the scope of this article is to extend the class of trees where the dispersion estimate holds. In the case of a graph with a closed path, in general there exist compact supported  eigenfunctions for the considered Laplace operator and then 
the dispersion estimate fails.

The motivation for studying  thin structures comes from mesoscopic physics and nanotechnology. Mesoscopic systems are those that have some dimensions which are too small to be treated using classical physics while they are too large to be considered on the quantum level only. The \textit{quantum wires} are physical systems with two dimensions reduced to a few nanometers. We refer to \cite{MR1937279} and references therein for more details on such type of structures.

The simplest model describing conduction in quantum wires is a Hamiltonian on a planar graph, i.e. a one-dimensional object. Throughtout the paper we consider a class of idealized quantum wires, where the configuration space is a planar graph and the Hamiltonian is minus the Laplacian with Kirchhoff's boundary conditions at the vertices of the graph. This condition makes the Hamiltonian to be a self-adjoint operator. More general coupling conditions that guarantee the self-adjointness are given in \cite{MR1671833}.

The problems addressed here enter in the framework of metric graphs or networks. Those are metric spaces which can be written as the union of finitely many intervals, which are compact or $[0,\infty)$ and any two of these intervals are either disjoint or intersect only in one or both of their endpoints.
Differential operators on metric graphs arise in a variety of applications. We mention some of them: carbon nano-structures \cite{MR2336365}, photonic crystals \cite{MR1856878}, high-temperature granular superconductors \cite{MR690539}, quantum waveguides \cite{carini}, 
free-electron theory of conjugated molecules in chemistry, quantum chaos, etc.
For more details we refer the reader to review papers \cite{MR1937279}, \cite{MR2042548}, \cite{gnut} and \cite{Exner}. 

The linear and cubic Schr\"odinger equation on simple networks with Kirchhoff connection conditions and particular type of data has been analyzed in \cite{MR2675854}. The symmetry imposed on the initial data and the shape of the networks allow to reduce the problem to a Schr\"odinger equation on the half-line with appropriate boundary conditions, for which a detailed study is done by inverse scattering. Some numerical experiments are also presented in \cite{MR2675854}. The propagation of solitons for the cubic Schr\"odinger equation on simple networks but with connection conditions in link with the mass and energy conservation is analyzed in \cite{zobi}. \\

The main result is the following, where by $\{I_e\}_{e\in E}$ we shall denote the edges of the tree.
\begin{theo}  \label{disp}
The solution of the linear Schr\"odinger equation on a tree is of the form
\begin{equation}\label{LSexpr}
e^{it\Delta_\Gamma}\bu_0(x)=\sum_{\lambda\in\mathbb R}\frac{a_\lambda}{\sqrt{|t|}} \int_{I_\lambda}e^{i\frac{\phi_\lambda(x,y)}{t}}\,\,\bu_0(y)\,dy.
\end{equation}
with $\phi_\lambda (x,y)\in\mathbb R$,  $I_\lambda \in\{I_e\}_{e\in E}$, $\sum_{\lambda\in\mathbb R} |a_\lambda|<\infty$, and it satisfies the
dispersion inequality 
\begin{equation}\label{dispersion}
\| e^{it\Delta_\Gamma}\bu_0\|_{\LL^\infty(\Gamma)}\leq \frac{C}{\sqrt {|t|}}
\| \bu_0\|_{\LL^1(\Gamma)},\,\,\, t\neq 0.
\end{equation}
\end{theo}

The proof uses the method in \cite{Valeria} in an appropriate way related to the ramified analysis on the tree, by recursion on the number of vertices. It consists in writing the solution in terms of the resolvent of the Laplacian, which in turn is computed in the framework of almost-periodic functions.\\

As mentioned before, Strichartz estimates \eqref{dsch0001} can be derived from the dispersion inequality and have been used intensively to obtain well-posedness results for the nonlinear Schr\"odinger equation (NSE). The arguments used in the context of NSE on $\mathbb R$ can also be used here to obtain the following as a typical result.

\begin{theo}\label{tree.nse}
 Let $p\in (0,4)$. For any $\bu_0\in \LL^2(\Gamma)$ there exists a unique solution
$$\displaystyle \bu\in C(\R,\LL^2(\Gamma))\cap \bigcap _{(q,r) \text{admissible}} \LL^{q}_{loc}(\R,\LL^r(\Gamma)),$$
of the nonlinear Schr\"odinger equation
\begin{equation}\label{eq.tree.non}
\left\{
\begin{array}{ll}
i\bu_t +\Delta_\Gamma \bu\pm|\bu|^{p}\bu=0,& t\neq 0,\\[10pt]
 \bu(0)=\bu_0,& t=0.
\end{array}
\right.
\end{equation}
Moreover, the $\LL^2(\Gamma)$-norm of $\bu$ is conserved along the time
$$\|\bu(t)\|_{\LL^2(\Gamma)}=\|\bu_0\|_{\LL^2(\Gamma)}.$$
\end{theo}
The proof is standard once the dispersion property is obtained and it follows as in \cite{1055.35003}, p. 109, Theorem 4.6.1.  \\

With the same method we obtain the same results in the case of the Laplacian on the graph with laminar coefficients (piecewise constants, bounded between two positive constants- the details on the laminar Laplacian are given in Section 3). This might be of physical interest when the wire on a edge is composed of different pieces. 
Equations with variable coefficients on networks have been previously analyzed in \cite{MR932369} for the heat equation and in \cite{MR1287844} for the wave equations. For clearness we prefer to treat separately the two cases even if the laminar case includes the constant coefficient case.\\


The paper is organized as follows. In Section \ref{notations} we introduce the Laplacian on a graph and write the systems we want to analyze. 
In Section \ref{const} we present in full details the proof of Theorem \ref{disp}. Section \ref{slaminar} contains the proof of the results of Theorem \ref{disp} in the laminar case. Some open problems are discussed in Section \ref{open}.

\section{Notations and Preliminaries}\label{notations}
In this section we present some generalities about metric graphs and introduce the Laplace operator on such structure.
Let $\Gamma=(V,E)$ be a graph where $V$ is a set of vertices and $E$ the set of edges. 
For each $v\in V$ we denote  $E_v=\{e\in E: v\in e\}$.
We assume that $\Gamma$ is a countable connected locally finite graph, i.e. the degree of each vertex $v$ of $\Gamma$ is finite:
$d(v)=|E_v|<\infty.$
The edges could be of finite length and then their ends are vertices of $V$ or they have infinite length and then we assume that each infinite edge is a ray with a single vertex belonging to $V$ (see \cite{MR2459876} for more details on graphs with infinite edges).

We fix an orientation of $\Gamma$ and for each oriented edge $e$, we denote by $I(e)$ the initial vertex and by $T(e)$ the terminal one. Of course in the case of infinite edges we have only initial vertices.

We identify every edge $e$ of $\Gamma$ with an interval $I_e$, where $I_e=[0,l_e]$ if the edge is finite and $I_e=[0,\infty)$ if the edge is infinite. This identification introduces a coordinate $x_e$ along the edge $e$. In this way $\Gamma$ is a metric space and is often named metric graph  \cite{MR2459876}. 

Let $v$ be a vertex of $V$ and $e$ be an edge in  $E_v$. We set for finite edges $e$
$$j(v,e)=\left\{
\begin{array}{lll}
0& \text{if} &v=I(e), \\[10pt]
l_e& \text{if} & v=T(e) 
\end{array}
\right.
$$
and
$$j(v,e)=0,\ \text{if}\ v=I(e)$$
for infinite edges.

We identify any function $\bu$ on $\Gamma$ with a collection $\{\ue\}_{e\in E}$ of functions $\ue$ defined on the edges  $e$ of $\Gamma$. Each $\ue$ can be considered as a function on the interval $I_e$. In fact, we use the same notation $\ue$ for both the function on the edge $e$ and the function on the interval $I_e$ identified with $e$.
For a function $\bu:\Gamma\rightarrow \cn$,  $\bu=\{u^e\}_{e\in E}$,  we denote by $f(\bu):\Gamma\rightarrow \cn$ the family 
$\{f(u^e)\}_{e\in E}$, where  $f(u^e):e\rightarrow\cn$.

A function $\bu=\{\ue\}_{e\in E}$ it is continuous if and only if $\ue$ is continuous on $I_e$ for every $e\in E$, and moreover, is continuous at the vertices of $\Gamma$:
$$\ue(j(v,e))=u^{e'}(j(v,e')), \quad \forall \ e,e'\in E_v.$$ 

The space $\LL^p(\Gamma)$, $1\leq p<\infty$ consists of all functions   $\bu=\{u_e\}_{e\in E}$ on $\Gamma$ that belong to $\LL^p(I_e)$
for each edge $e\in E$ and 
$$\|\bu\|_{\LL^p(\Gamma)}^p=\sum _{e\in E}\|u^e\|_{\LL^p(I_e)}^p<\infty.$$
Similarly, the space $\LL^\infty(\Gamma)$ consists of all functions that belong to $\LL^\infty(I_e)$ for each edge $e\in E$ and
$$\|\bu \|_{\LL^\infty(\Gamma)}=\sup _{e\in E}\|u^e\|_{\LL^\infty(I_e)}<\infty.$$

The Sobolev space $H^m(\Gamma)$, $m\geq 1$ an integer, consists in all continuous functions on $\Gamma$ that  belong to
$H^m(I_e)$ for each $e\in E$ and 
$$\|\bu \|_{H^m(\Gamma)}^2=\sum _{e\in E}\|u^e\|_{H^m(e)}^2<\infty.$$
The above spaces are  Hilbert spaces with the inner products
$$(\bu,\bv)_{\LL^2(\Gamma)}=\sum _{e\in E}(\ue,\ve)_{\LL^2(I_e)}=\sum _{e\in E}\int _{I_e}\ue(x)\overline{\ve}(x)dx$$
and
$$(\bu,\bv)_{H^m(\Gamma)}=\sum _{e\in E}(\ue,\ve)_{H^m(I_e)}= \sum _{e\in E}\sum _{k=0}^m\int _{I_e} \frac{d^k\ue}{dx^k} \overline{\frac{d^k\ve}{dx^k} }dx.$$

We now introduce the Laplace operator $\Delta_\Gamma$ on the graph $\Gamma$. Even if it is a standard procedure we prefer for the sake of completeness to follow  \cite{MR1476363}.  Consider the sesquilinear continuous form $\varphi$
on $H^1(\Gamma)$ defined by
$$\varphi(\bu,\bv)=(\bu_x,\bv_x)_{L^2(\Gamma)}=\sum _{e\in E}\int _{I_e}\ue_x(x)\overline{\ve_x}(x)dx.$$
We denote by $D(\Delta_\Gamma)$ the set of all the functions $\bu\in H^1(\Gamma)$ such that the linear map $\bv\in H^1(\Gamma)\rightarrow \varphi_{\bu}(\bv)=\varphi(\bu,\bv)$ satisfies
$$|\varphi(\bu,\bv)|\leq C\|\bv\|_{L^2(\Gamma)}\quad \text{for all}\ \bv\in H^1(\Gamma).$$
For $\bu\in D(\Delta_\Gamma)$, we can extend $\varphi_\bu$ to a linear continuous mapping on $\LL^2(\Gamma)$. There is a unique element in  $L^2(\Gamma)$ denoted by $\Delta_\Gamma \bu$, such that, 
$$\varphi(\bu,\bv)=-(\Delta_\Gamma \bu,\bv)\quad \text{for all}\ \bv\in H^1(\Gamma).$$

We now define the normal exterior derivative of a function $\bu=\{u^e\}_{e\in E}$  at the endpoints of the edges.
For each $e\in E$ and $v$ an endpoint of $e$ we consider the normal derivative of the restriction of $\bu$ to the edge $e$ of $E_v$ evaluated at $i(v,e)$ to be defined by:
$$\frac {\partial u^e}{\partial n_e}(j(v,e))=
\left\{
\begin{array}{lll}
-u_x^e(0+)&\text{if}& j(v,e)=0, \\[10pt]
u_x^e(l_e-)& \text{if}&j(v,e)=l_e  .
\end{array}
\right.
$$
With this notation it is easy to characterise $D(\Delta_\Gamma)$ (see \cite{MR1476363}): 
$$D(\Delta_\Gamma)=\Big\{\bu=\{u^e\}_{e\in E}\in H^2(\Gamma): \sum _{e\in E_v} \frac {\partial u^e}{\partial n_e}(j(v,e))=0\quad 
\text{for all}\ v\in V\Big\}$$
and
$$(\Delta_\Gamma \bu)^e=(u^e)_{xx}\quad \text{for all}\ e\in E, \bu\in D(\Delta_\Gamma).$$
In other words  $D(\Delta_\Gamma)$ is the space of all continuous functions on $\Gamma$, $\bu=\{u^e\}_{e\in E}$,  such that for every edge $e\in E$,
$u^e\in H^2(I_e)$,  and satisfying the following Kirchhoff-type condition:
$$\sum _{e\in E: T(e)=v} u^e_x(l_e-)-\sum _{e\in E: I(e)=v}u_x^e(0+)=0 \quad \text{for all} \ v\in V.$$
It is easy to verify that $(\Delta_\Gamma, D(\Delta_\Gamma ) )$ is a linear, unbounded, self-adjoint, dissipative operator on $\LL^2(\Gamma)$, i.e. 
$\Re (\Delta_\Gamma\bu,\bu)_{\LL^2(\Gamma)}\leq 0$ for all $\bu\in D(\Delta_\Gamma)$. 

Let us consider the LSE on  $\Gamma$:
\begin{equation}\label{eq.tree}
\left\{
\begin{array}{ll}
i\bu _t(t,x)+\Delta_\Gamma \bu(t,x)=0,& x\in \Gamma, t\neq 0 ,\\[10pt]
\bu(0)=\bu_0,&   x\in \Gamma.
\end{array}
\right.
\end{equation}

Using  the properties of the  operator $i\Delta_\Gamma$ 
 we obtain as a consequence of the Hille-Yosida theorem the following well-posedness result.
\begin{theo}\label{existence}
For any $\bu_0\in D(\Delta_\Gamma)$ there exists a unique solution $\bu(t)$ of system \eqref{eq.tree} that satisfies
$$\bu\in C(\R,D(\Delta_\Gamma))\cap C^1(\R,\LL^2(\Gamma)).$$
Moreover, for any $\bu_0\in \LL^2(\Gamma)$, there exists a unique solution $\bu\in C(\R,\LL^2(\Gamma))$ that satisfies
$$\|\bu(t)\|_{\LL^2(\Gamma)}=\|\bu_0\|_{\LL^2(\Gamma)}\quad \text{for all}\ t\in \R.$$
\end{theo}

The $\LL^2(\Gamma)$-isometry property is a consequence of the fact that the operator $i\Delta_\Gamma$ satisfies
$\Re(i\Delta_\Gamma \bu,\bu)_{\LL^2(\Gamma)}=0$ for all $\bu\in D(\Delta _\Gamma)$. \\

For any $\bu_0\in D(\Delta_\Gamma)$ system \eqref{eq.tree} can be written in an explicit way as follows: 
\begin{equation}\label{system1exp}
\left\{
\begin{array}{ll}
u^e\in C(\R,H^2(I_e))\cap C^1(\R,L^2(I_e)),& e\in E,\\[10pt]
iu^e_t(t,x)+\Delta u^e(t,x)=0,& x\in I_e, t\neq 0, \\[10pt]
\text{for all} \ v\in V, \ u^{e}(t,j(v,e))=u^{e'}(t,j(v,e')),& \forall e,e'\in E_v, t\neq 0,\\[10pt]
\displaystyle\sum _{e\in E: T(e)=v} u^e_x(t,l_e-)-\sum _{e\in E: I(e)=v}u_x^e(t,0+)=0 & \text{for all} \ v\in V.
\end{array}
\right.
\end{equation}

\medskip

Let us now consider the laminar case. We consider $\sigma$ a piecewise constant function on each edge of the tree such that there exist two positive constants $\sigma_1$ and $\sigma_2$ such that
$$0<\sigma_1<\sigma (x)<\sigma_2, \quad \forall\ x\in I_e, \ \forall e\in E. $$ 
With a similar argument as before we introduce the operator $\Delta_{\sigma,\Gamma}$\ as follows
$$D(\Delta_{\sigma,\Gamma})=\Big\{\bu=\{u^e\}_{e\in E}\in H^2(\Gamma): \sum _{e\in E_v} \sigma (j(v,e))\frac {\partial u^e}{\partial n_e}(j(v,e))=0\quad 
\text{for all}\ v\in V\Big\}$$
and
$$(\Delta_{\sigma,\Gamma} \bu)^e=\partial_x(\sigma (\cdot)\partial _x(u^e))\quad \text{for all}\ e\in E, \bu\in D(\partial_x (\sigma \partial _x)).$$

It follows that for any $\bu_0\in  D(\Delta_{\sigma,\Gamma})$ the following system is well-posed
\begin{equation}\label{system2}
\left\{
\begin{array}{ll}
u^e\in C(\R,H^2(I_e))\cap C^1(\R,\LL^2(I_e)),& e\in E,\\[10pt]
iu^e_t(t,x)+\partial _x(\sigma (x)\partial _x u^e)(t,x)=0,& x\in I_e, t\neq 0, \\[10pt]
\text{for all} \ v\in V, \ u^{e}(t,j(v,e))=u^{e'}(t,j(v,e')),& \forall e,e'\in E_v, t\neq 0,\\[10pt]
\displaystyle\sum _{e\in E: T(e)=v} \sigma (l_e-)u^e_x(t,l_e-)=\sum _{e\in E: I(e)=v}\sigma (0+)u_x^e(t,0+) & \text{for all} \ v\in V.
\end{array}
\right.
\end{equation}
We remark that when function $\sigma$ is identically equal to one we obtain the previous system \eqref{system1exp}.

\section{The constant coefficient case}\label{const}
\subsection{The description of the solution}

For $\omega\geq 0$ let $R_\omega$ 
be the resolvent of the Laplacian on a tree
$$R_{\omega}\bold f=(-\Delta_{\Gamma}+\omega ^2I)^{-1}\bold f.$$

We shall prove in Lemma \ref{ancontres} that $\omega R_\omega \bold f(x)$ can be analytically continued in a region containing the imaginary axis. 
Therefore we can use a spectral calculus argument to write the solution of the Schr\"odinger equation with initial data $\bu_0$ as
\begin{equation}\label{SL}
e^{it\Delta_\Gamma}\bu_0(x)=
\int_{-\infty}^{\infty}e^{it\tau^2}\tau R_{i\tau}\bu_0(x)\frac{d\tau}{\pi}.
\end{equation}
We shall also obtain in Lemma \ref{ressumlemma} that the following decomposition holds
\begin{equation}\label{sumexp}
\tau R_{i\tau} \bu_0(x)=
\sum_{\lambda\in\mathbb R} b_\lambda e^{i\tau \psi_\lambda(x)}\int_{I_\lambda}\bu_0(y)
e^{i\tau \beta_\lambda y}dy,
\end{equation}
with $\psi_\lambda (x),\beta_\lambda\in\mathbb R$,  $I_\lambda\in\{I_e\}_{e\in E}$ and $\sum_{\lambda\in\mathbb R} |b_\lambda|<\infty$. Then decomposition \eqref{LSexpr} is implied by \eqref{SL}, \eqref{sumexp} and the fact that for $t>0$ and $r\in\mathbb R$
$$\int_{-\infty}^{\infty}e^{it\tau^2}e^{i\tau r}d\tau=e^{i\frac\pi 4}\sqrt\pi\frac{e^{-\frac{r^2}{4t}}}{\sqrt{t}}.$$
From \eqref{LSexpr} the dispersion estimate \eqref{dispersion} of Theorem \ref{disp} follows immediately since $\sum_{\lambda\in\mathbb R} |\alpha_\lambda|<\infty$.

Above and in what follows the integration of function $\bold {f}=(f^e)_{e\in E}$ on interval $I_e$ means the integral of $f^e$ on the considered interval. 

\begin{rem}
As in \cite{Valeria} we notice that since we can express the solution of the wave equation $\bold v_{tt}-\Delta_\Gamma \bold v=0$ with initial data $(\bold v_0,0)$ as
$$\bold v(t,x)=\int_{-\infty}^{\infty}e^{it\tau}R_{i\tau}\bold v_0(x)i\tau
\frac{d\tau}{2\pi},$$
the property
$$ \sup_{x\in\Gamma} \int_{-\infty}^{\infty}\mo{\bold v(t,x)}dt\leq C\|
\bold v_0\|_{\LL^1(\Gamma)}$$
follows similarly. Let us mention here that the wellposedness of  a class of nonlinear dispersive waves on trees, the Benjamin-Bona-Mahony equation, has been investigated in \cite{MR2500096}.
\end{rem}

\subsection{Structure of the resolvent}

In order to obtain the expression of the resolvent second-order equations 
$$(R_\omega \bold f)''=\omega^2 R_\omega \bold f-\bold f$$ 
must be solved on each edge of the tree together with coupling conditions at each vertex. Then, on each edge parametrized by $I_e$,
$$R_\omega \bold f(x)=ce^{\omega x}+\tilde ce^{-\omega
x}+\frac{1}{2\omega}\int_{I_e}\bold f(y)\,e^{-\omega
\mo{x-y}}dy, \ x\in I_e.$$
Since $R_\omega \bold f$ belongs to $\LL^2(\Gamma)$ the coefficients $c$'s are zero on the  infinite edges $e\in\mathcal E$, parametrized by $[0,\infty)$. If we denote by $\mathcal I$ the set of internal edges, we have $2|\mathcal I|+|\mathcal E|$ coefficients. The Kirchhoff conditions of continuity of $R_\omega \bold f$ and 
of transmission of $\partial_xR_\omega \bold f$ at the vertices of the tree give the system of 
equations on the coefficients. We have the same number of equations as the number of unknowns. We denote $D_\Gamma$ the matrix of the system, whose elements are real powers of $e^{\omega}$.

Therefore the resolvent $R_\omega \bold f(x)$ is a finite sum of terms :  
\begin{equation}\label{ressum}
R_\omega \bold f(x)=\frac{1}{\omega\,\det D_\Gamma(\omega)}
\sum_{\lambda=1}^{N(\Gamma)} c_\lambda e^{\pm\omega \Phi_\lambda(x)}\int_{I_\lambda} \bold f(y)
\,e^{\pm\omega y}dy+\frac{1}{2\omega}\int_{I_e}\bold f(y)\,e^{-\omega \mo{x-y}}dy,
\end{equation}
where  $x\in I_e$, $\Phi_\lambda (x)\in\mathbb R$, $I_\lambda \in\{I_e\}_{e\in E}$ and $|N(\Gamma)|<\infty$. We shall prove the following proposition that will imply Lemma \ref{ancontres} and \ref{ressumlemma} needed for obtaining Theorem \ref{disp}.

\begin{prop}\label{propdet}Function $\det D_\Gamma(\omega)$ is lower bounded by a positive constant on a strip containing the imaginary axis:
$$\exists c_\Gamma,\epsilon_\Gamma>0,\,\,|\det D_\Gamma(\omega)|>c_\Gamma,\forall\omega\in\mathbb C, |\Re\omega|<\epsilon_\Gamma.$$
\end{prop}

\begin{lemma}\label{ancontres} Function $\omega R_\omega \bold f(x)$ can be analytically continued in a region containing the imaginary axis.
\end{lemma}
\begin{proof}The proof is an immediate consequence of decomposition \eqref{ressum} and of Proposition \ref{propdet}.
\end{proof}

\begin{lemma}\label{ressumlemma}The following decomposition holds
$$\tau R_{i\tau} \bu_0(x)=
\sum_{\lambda\in\mathbb R} b_\lambda e^{i\tau \psi_\lambda(x)}\int_{I_\lambda}\bu_0(y)
e^{i\tau \beta_\lambda y}dy,$$
with $\psi_\lambda (x),\beta_\lambda\in\mathbb R$,  $I_\lambda \in\{I_e\}_{e\in E}$ and $\sum_{\beta\in\mathbb R} |b_\lambda|<\infty$.
\end{lemma}

\begin{proof}
We notice that for $\tau\in\mathbb R$, $\det D_\Gamma(i\tau)$ is a finite sum of powers of $e^{i\tau}$. Then, by Proposition \ref{propdet} we are in the framework of  a classical theorem in representation theory (\S29, Cor.1 of \cite{Ge}) that asserts that the inverse of $\det D_\Gamma(i\tau)$ is $\sum_{\lambda\in\mathbb R} d_\lambda e^{i\tau \lambda}$ with $\sum_{\lambda\in\mathbb R} |d_\lambda|<\infty$, and from \eqref{ressum} the Lemma follows.
\end{proof}

The rest of this section is the proof of Proposition \ref{propdet}. We shall show by recursion on the number of vertices the following stronger ``double" property:
$$P(n):\,\,\,\,\mbox{If }\Gamma \mbox{ has $n$ vertices, we have the property }\mathcal P,$$
$$\mathcal P : \exists c_\Gamma,\epsilon_\Gamma>0,\exists 0<r_\Gamma<1,\,\,|\det D_\Gamma(\omega)|>c_\Gamma,\,\left|\frac{\det \tilde D_\Gamma(\omega)}{\det D_\Gamma(\omega)}\right|<r_{\Gamma},\,\forall\omega\in\mathbb C, |\Re\omega|<\epsilon_\Gamma.$$
We have denoted by $\tilde D_\Gamma(\omega)$ the matrix of the system verified by the coefficients, if we impose that on one of the last infinite edges  $l\in\mathcal E$ we replace in the expression of the resolvent $\tilde ce^{-\omega x}$ by $ce^{\omega x}$.

\subsection{Proof of $P(1)$} In this case we have  a star-shaped tree with  $m\geq 3$ of edges.
All the edges are parametrized by $[0,\infty)$. In particular $D_\Gamma(\omega)=D_\Gamma$. We shall actually prove a stronger property, which implies the property $\mathcal P$ for any $\epsilon_\Gamma>0$:
$$P(1,m): \,\,\,\,\mbox{If }\Gamma \mbox{ has 1 vertex and $m$ edges, }\det D_\Gamma(\omega)=m\, \mbox{and} \, \det \tilde D_\Gamma(\omega)=m-2.$$

The resolvent contains  $\tilde c_{j}e^{-\omega x}$, $1\leq j\leq m$, on each of the external edges. We write matrix $D_{\Gamma_m}$ such that the last line is coming from  Kirchhoff derivative condition, and that the other lines are coming from  Kirchhoff continuity condition. So matrix $D_{\Gamma_m}$ can be written such that it has components $1$ on the last line, and on the principal diagonal,  $d_{i,i+1}=-1$ for $1\leq j\leq m-1$, and zeros elsewhere
$$D_{\Gamma^m}=\left(\begin{array}{cccccccc}1 & -1 & &&&&&\\  & 1 & -1&&&&&\\ &&.&.&&&&\\&&&.&.&&&\\&&&&.&.&&\\&&&&&1&-1&\\&&&&&&1&-1\\1&1&1&1&1&1&1&1 \end{array}\right).$$
By developement with respect to the last column and $P(1,m-1)$,
$$\det D_{\Gamma^m}=1+\det D_{\Gamma^{m-1}}=m.$$
Similarly,
$$\det \tilde D_{\Gamma^m}=\left|\begin{array}{cccccccc}1 & -1 & &&&&&\\  & 1 & -1&&&&&\\ &&.&.&&&&\\&&&.&.&&&\\&&&&.&.&&\\&&&&&1&-1&\\&&&&&&1&-1\\1&1&1&1&1&1&1&-1 \end{array}\right|=1+\det  \tilde D_{\Gamma^{m-1}}=m-2,$$
so $P(1,m)$ is proven for any $m\geq 3$ and implicitly $P(1)$.

\subsection{Proof of $P(n-1)\Rightarrow P(n)$.}\label{rec} Any tree $\Gamma_n$ with $n$ vertices, $n\geq 2$, can be seen as a tree $\Gamma_{n-1}$ with $n-1$ vertices on which we add an extra-vertex. More precisely, let us consider a vertex $v$ from which  there start $m\geq 2$ external infinite edges and one internal edge connecting it to the rest of the tree (see Fig. \ref{fig:1}). {Let us notice that such a choice is possible since the graph has no cycles}. In particular the edge whose lower extremity is this vertex $v$ is an internal edge $l$, whose length should be denoted by $a$, and whose upper vertex we denote by $\tilde v$. Now we remove this vertex and transform the internal edge $l$ into an external infinite one. The new graph $\Gamma_{n-1}$ has $n-1$ vertices. 

\begin{figure}[t]
\includegraphics[width=6cm]{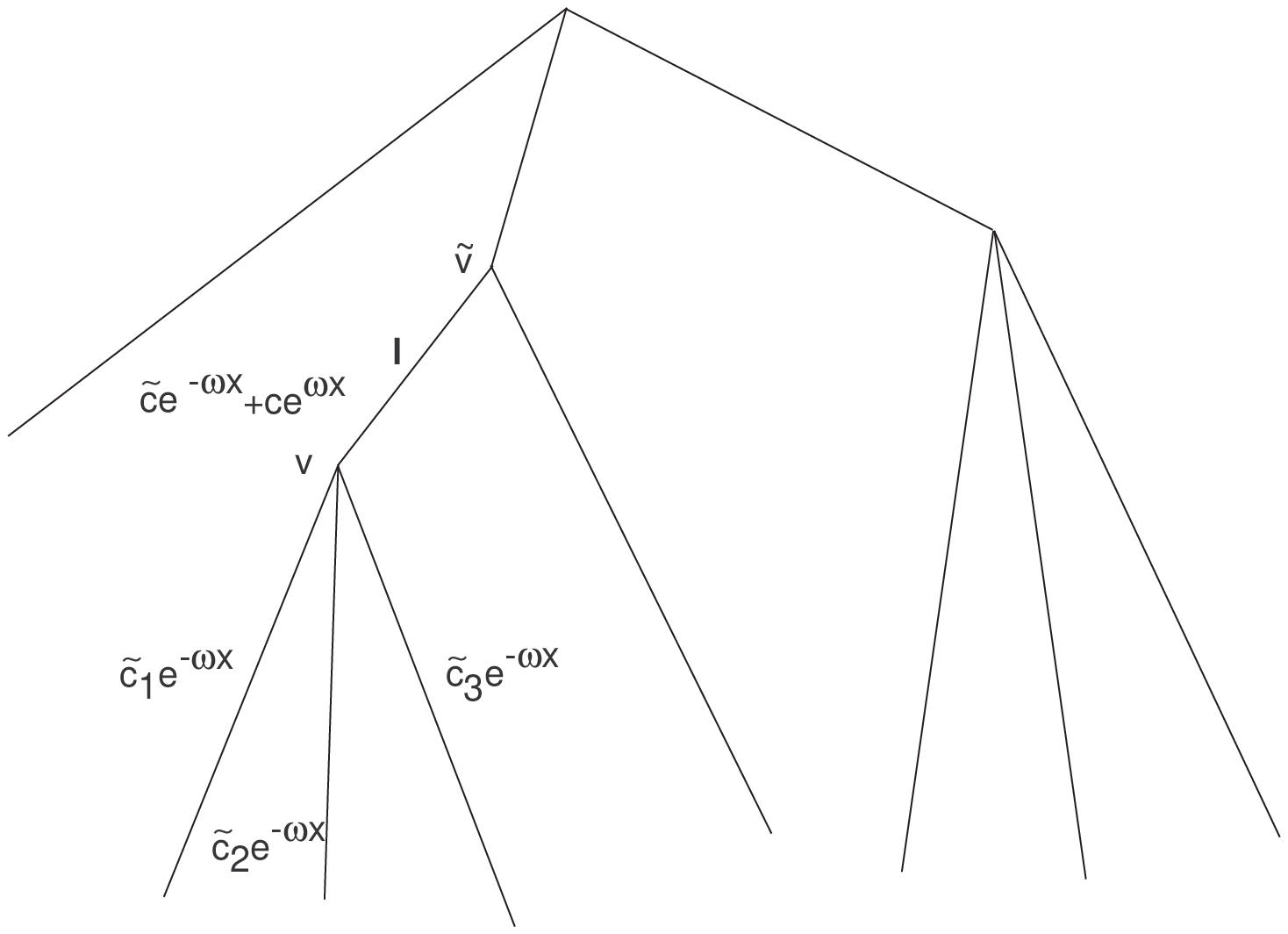}
\includegraphics[width=6cm]{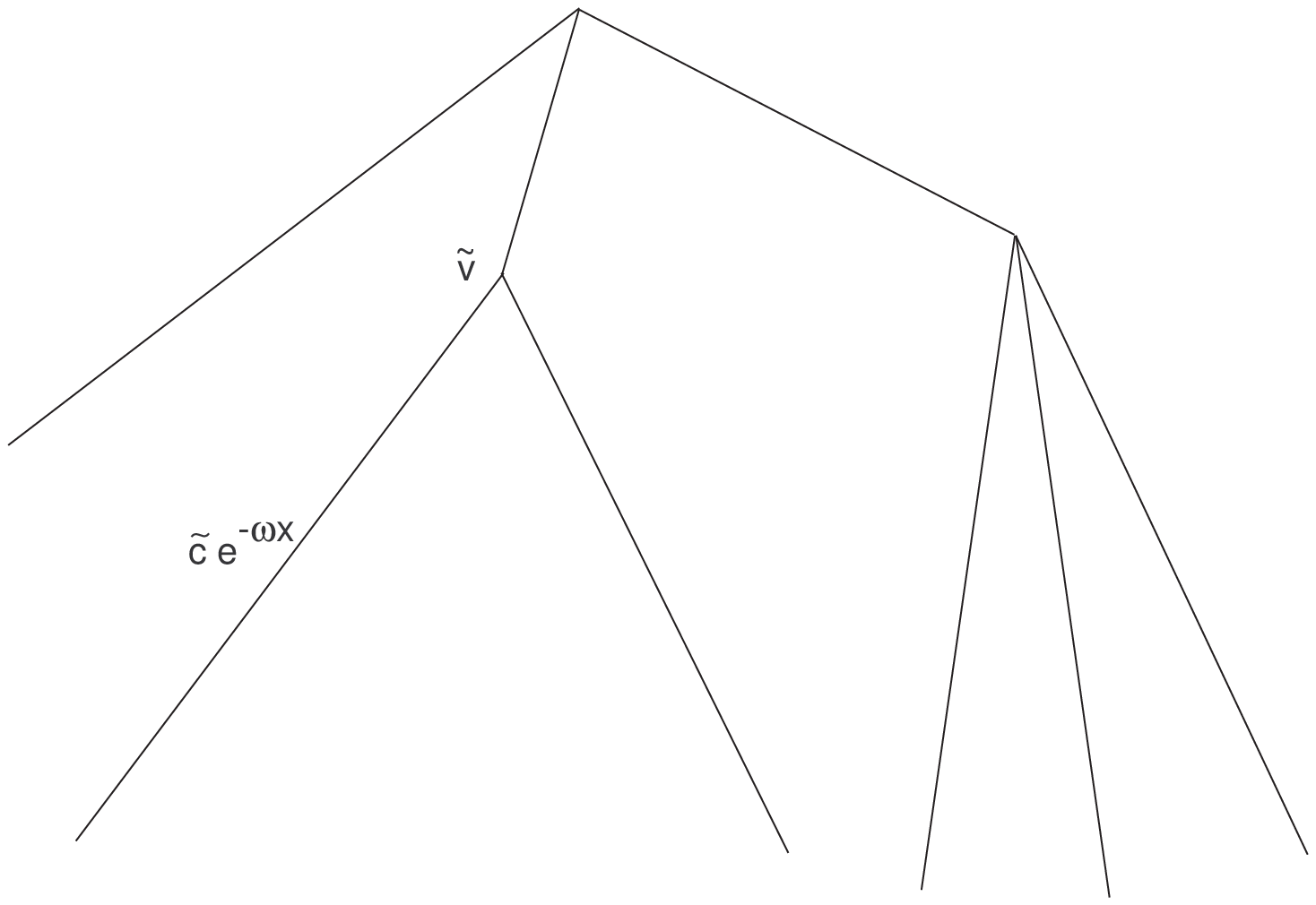}
\caption{With vertex $v$ we obtain tree $\Gamma_{4}$ (left) from $\Gamma_3$ (right).\label{fig:1}}
\end{figure}

With respect to the problem on $\Gamma_{n-1}$, the resolvent on $\Gamma_n$ involves a new term $ ce^{\omega x}$ aside from $\tilde ce^{-\omega x}$ on the interval edge $l$, and on the external edges emerging from the vertex $v$ it involves terms $\tilde c_{j}e^{-\omega x}$, $1\leq j\leq m$. We have also the Kirchhoff conditions at the vertex $v$, which give $m+1$ equations on the coefficients. 

We write the square $N\times N$ matrix $D_{\Gamma_n}$ such that the last $m+2$ column corresponds to the unknowns $\tilde c, c, \tilde c_1,...,\tilde c_m$ . On the last line we write the Kirchhoff derivative condition at the vertex $v$, and on the $N-j$ lines, $1\leq j\leq m$ the Kirchhoff continuity conditions at the vertex $v$. Also, on the $N-m-1$ line we write the derivative condition in the vertex $\tilde v$ and on the $N-m-2$ line the continuity condition in $\tilde v$ relating $\tilde c$, and now also $c$, to the others coefficients.   So $D_{\Gamma_n}$ is a matrix obtained from the $(N-m-1)\times (N-m-1)$ matrix $D_{\Gamma_{n-1}}$ (whose last column corresponds to the unknown $\tilde c$) and from $D_{\Gamma^m}$ (see previous subsection) in the following way
$$(D_{\Gamma_n})_{i,j}=
\left\{\begin{array}{c}
(D_{\Gamma_{n-1}})_{i,j}, 1\leq i,j\leq N-m-1,\\[7pt]
-1, (i,j)\in\{(N-m-2,N-m),(N-m-1,N-m),(N-m,N-m+1)\}, \\[7pt]
e^{\omega a}, (i,j)\in\{(N-m,N-m),(N,N-m)\}, \\[7pt]
e^{-\omega a}, (i,j)=(N-m,N-m-1),\\[7pt]
-e^{-\omega a}, (i,j)=(N,N-m-1),\\[7pt]
(D_{\Gamma^m})_{i,j}, N-m+1\leq i,j\leq N,
\end{array}\right.$$

$$D_{\Gamma_n}=
\left(\begin{array}{cccc|ccccccccc}
&&&&  && & &&&&&\\
&& {D_{\Gamma_{n-1}}}&&&& & &&&&&\\
&&&&-1&& & &&&&&\\
&&&&-1&& & &&&&&\\ \hline
&&&e^{-\omega a}&e^{\omega a}& -1 & &&&&&\\ 
&&&&&1 & -1 & &&&&&\\  &&&&&& 1 & -1&&&&&\\ 
&&&&&&&.&.&&&&\\&&&&&.&.&&&&&\\&&&&&&&&&.&.&&\\&&&&&&&&&&1&-1&\\&&&&&&&&&&&1&-1\\&&&-e^{-\omega a}&e^{\omega a}&1&1&1&1&1&1&1&1 \end{array}\right).$$
In the case $n=1$ we have that $D_{\Gamma_1}=D_{\Gamma^{\tilde m}}$ for some $\tilde m\geq 3$. Also, we emphasize that with the above recursion, matrix $\tilde D_{\Gamma_n}$ is obtained from $D_{\Gamma_n}$ by replacing its $N\times N$ element with $-1$.

We develop $\det D_{\Gamma_n}$ with respect to the last $m+1$ lines, that is as an alternated sum of determinants of $m+1\times m+1$ minors composed from the last $m+1$ lines of $D_{\Gamma_n}$ times the determinant of $D_{\Gamma_n}$ without the lines and columns the minor is made of. The only possibility to obtain a  $m+1\times m+1$ minor composed from the last $m+1$ lines of $D_{\Gamma_n}$ different from zero is to choose one of the columns $N-m-1$ and $N-m$, together with all last $m$ columns. This follows from the fact that if we eliminate from $\det D_{\Gamma_n}$ both columns $N-m-1$ and $N-m$, together with $m-1$ columns among the last $m$ columns, we obtain a block-diagonal type matrix, with first diagonal block $D_{\Gamma_{n-1}}$ with its last column replaced by zeros, so its determinant vanishes. Therefore
$$\det D_{\Gamma_n}=\det D_{\Gamma_{n-1}}\left|\begin{array}{ccccccccc}e^{\omega a}& -1 & &&&&&\\&1 & -1 & &&&&&\\  && 1 & -1&&&&&\\ &&&.&.&&&&\\&&&&.&.&&&\\&&&&&.&.&&\\&&&&&&1&-1&\\&&&&&&&1&-1\\e^{\omega a}&1&1&1&1&1&1&1&1 \end{array}\right|$$
$$-\det \tilde D_{\Gamma_{n-1}}\left|\begin{array}{ccccccccc}e^{-\omega a}& -1 & &&&&&\\&1 & -1 & &&&&&\\  && 1 & -1&&&&&\\ &&&.&.&&&&\\&&&&.&.&&&\\&&&&&.&.&&\\&&&&&&1&-1&\\&&&&&&&1&-1\\-e^{-\omega a}&1&1&1&1&1&1&1&1 \end{array}\right|.$$
By developing with respect to the first column the $m+1\times m+1$ minors,
$$\det D_{\Gamma_n}=\det D_{\Gamma_{n-1}}(e^{\omega a}\det D_{\Gamma^m}+(-1)^{m+2}e^{\omega a}(-1)^m)$$
$$-\det \tilde D_{\Gamma_{n-1}}(e^{-\omega a}\det D_{\Gamma^m}-(-1)^{m+2}e^{-\omega a}(-1)^m),$$
so using from the previous subsection that $\det D_{\Gamma^m}=m,\det \tilde D_{\Gamma^m}(\omega)=m-2$, we find 
$$\det D_{\Gamma_n}(\omega)=(m+1)e^{\omega a}\det D_{\Gamma_{n-1}}(\omega)-(m-1)e^{-\omega a}\det \tilde D_{\Gamma_{n-1}}(\omega)$$
$$=(m+1)e^{\omega a}\det D_{\Gamma_{n-1}}(\omega)\left(1-e^{-2\omega a}\frac{m-1}{m+1}\frac{\det \tilde D_{\Gamma_{n-1}}(\omega)}{\det D_{\Gamma_{n-1}}(\omega)}\right).$$

Now, from $P(n-1)$ we have for $|\Re \omega|$ small enough 
$$1-e^{-2\omega a}\frac{m-1}{m+1}\frac{\det \tilde D_{\Gamma_{n-1}}(\omega)}{\det D_{\Gamma_{n-1}}(\omega)}>c_0>0.$$
Also, $P(n-1)$ gives us the existence of two positive constants  $ c_{\Gamma_{n-1}}$ and $\epsilon_{\Gamma_{n-1}}$ such that $|\det D_{\Gamma_{n-1}}(\omega)|>c_{\Gamma_{n-1}},\,\forall\omega\in\mathbb C, |\Re\omega|<\epsilon_{\Gamma_{n-1}},$ so eventually we get 
$$\exists c_{\Gamma_{n}},\epsilon_{\Gamma_{n}}>0,\,|\det D_{\Gamma_{n}}(\omega)|>c_{\Gamma_{n}},\,\forall\omega\in\mathbb C, |\Re\omega|<\epsilon_{\Gamma_{n}},$$
and the first part of property $\mathcal P$  is proved for $P(n)$. 

In a similar way we get
$$\det \tilde D_{\Gamma_n}(\omega)=(m-1)e^{\omega a}\det D_{\Gamma_{n-1}}(\omega)-(m-3)e^{-\omega a}\det \tilde D_{\Gamma_{n-1}}(\omega),$$
so
$$\frac{\det \tilde D_{\Gamma_{n}}(\omega)}{\det D_{\Gamma_{n}}(\omega)}=\frac{\frac{m-1
}{m+1}-\frac{m-3
}{m+1}e^{-2\omega a}\frac{\det \tilde D_{\Gamma_{n-1}}(\omega)}{\det D_{\Gamma_{n-1}}(\omega)}}{1-\frac{m-1
}{m+1}e^{-2\omega a}\frac{\det \tilde D_{\Gamma_{n-1}}(\omega)}{\det D_{\Gamma_{n-1}}(\omega)}}.$$
Thus  we also get the second part of $\mathcal P$ for $P(n)$ since 
$$\left|\frac{\frac{m-1
}{m+1}-\frac{m-3
}{m+1}z}{1-\frac{m-1
}{m+1}z}\right|<1\iff 0<(m-2)(|z|^2-1)+2(m-1)(1-\Re z).$$

\section{The laminar coefficient case}\label{slaminar}

For $\omega\geq 0$ let $R_\omega$ 
be the resolvent of the operator $\Delta_{\sigma,\Gamma}$ defined in Section \ref{notations}
$$R_{\omega}\bold f=(-\Delta_{\sigma,\Gamma}+\omega ^2I)^{-1}\bold f.$$ 
We shall proceed as in the previous section, and the main point will be the proof of Proposition \ref{propdet} in the laminar coefficient case. On each side of the edge, parametrized by $x\in I\subset\mathbb R$, where the coefficient in the laminar Laplacian is $\sigma(x)=\frac{1}{b^2}$, the resolvent writes
$$R_\omega \bold f(x)=ce^{\omega bx}+\tilde ce^{-\omega
bx}+\frac{1}{2\omega}\int_{I}\bold f(y)\,e^{-\omega b
\mo{x-y}}dy, \ x\in I.$$
As in the previous section, and using the same notations, we shall show by recursion on the number of vertices property $P(n)$ which leads to the dispersion estimate.

\subsection{Proof of $P(1)$} We prove property $P(1)$ on a star-shaped tree by recursion on the number of discontinuities in the laminar structure:
$$P(1,p): \,\,\,\,\mbox{If }\Gamma \mbox{ has 1 vertex and $p$ discontinuities along its edges we have  property }\mathcal P.$$
We denote by $m\geq 3$ the number of edges.\\

We start with $P(1,0)$. We denote by $\frac{1}{b_j^2}$, $1\leq j\leq m$ the coefficients of the laminar Laplacian on each edge.  
The resolvent contains the
terms  $\tilde c_je^{-\omega b_jx}$, $1\leq j\leq m$ (and $c_me^{\omega b_3x}$ on the last edge for the computation of $\tilde D_\Gamma$) on each edge. We have
$$D_{\Gamma^{m,0}}=\left(\begin{array}{cccccccc}1 & -1 & &&&&&\\  & 1 & -1&&&&&\\ &&.&.&&&&\\&&&.&.&&&\\&&&&.&.&&\\&&&&&1&-1&\\&&&&&&1&-1\\\frac{1}{b_1}&\frac{1}{b_2}&\frac{1}{b_3}&..&..&\frac{1}{b_{m-2}}&\frac{1}{b_{m-1}}&\frac{1}{b_m} \end{array}\right).$$
By developing with respect to the last column, 
$$\det D_{\Gamma^{m,0}}=\frac{1}{b_m}+\det D_{\Gamma^{m-1,0}}=\sum_{j=1}^m\frac{1}{b_j}.$$
Similarly we obtain
$$\det \tilde D_{\Gamma^{m,0}}=\sum_{j=1}^{m-1}\frac{1}{b_j}-\frac{1}{b_m},$$
so the property $\mathcal P$ follows immediately.

Now we shall prove that $P(1,p-1)$ implies $P(1,p)$. Without loss of generality we can suppose that on the last $m$-th edge of $\Gamma^{m,p}$ there is at least one discontinuity. We denote by $x_{f}$ the last discontinuity point on this edge, and by  $x_{i}$ the previous discontinuity if there is one, or $x_i=0$ otherwise. We denote $a$ the length $x_f-x_i$, by $\frac{1}{b_i^2}$ the coefficient of the laminar Laplacian on $(x_i,x_f)$ and by $\frac{1}{b_f^2}$ the coefficient of the laminar Laplacian on $(x_f,\infty)$. 
We call $\Gamma^{m,p-1}$ the graph obtained from $\Gamma^{m,p}$ by removing the last discontinuity $x_f$ on the last edge, and we extend the laminar Laplacian on it on $[x_f,\infty)$ by $\frac{1}{b_i^2}$. 

With respect to the problem on $\Gamma^{m,p-1}$, in the expression of the resolvent we have on $[x_i,x_f]$ aside from the term $\tilde c_ie^{-\omega b_ix}$, the extra term $c_ie^{\omega b_ix}$, and on $[x_f,\infty)$ a term $\tilde c_fe^{-\omega b_f x}$. Also, there are two connection conditions at the new discontinuity point $x_f$. We write the matrix $D_{\Gamma^{m,p}}$ such that the last three columns correspond to the unknowns $(\tilde c_i,c_i,\tilde c_f)$, the last two lines come from the connection condition at $x_f$, and the previous last two lines come from the connection condition at $x_i$, if $x_i$ is a discontinuity point, and from the Kirchhoff conditions concerning $c_i,\tilde c_i$ if it is the vertex.

We have
$$D_{\Gamma^{m,p}}=\left(\begin{array}{cccc|ccccccccc}&&&&&& & \\&&D_{\Gamma^{m,p-1}}&&&& & \\&&&&-e^{\omega b_i x_i}&& & \\&&&&-\frac{e^{\omega b_i x_i}}{b_i}&& &\\ \hline &&&e^{-\omega b_i x_f}&e^{\omega b_i x_f}& -e^{-\omega b_f x_f} \\&&&-\frac{e^{-\omega b_i x_f}}{b_i}&\frac{e^{\omega b_i x_f}}{b_i}&-\frac{e^{-\omega b_f x_f}}{b_f}\end{array}\right).$$
By developing with respect to the last two lines,
$$\det D_{\Gamma^{m,p}}= e^{\omega b_i x_f} e^{-\omega b_f x_f}\left(\frac{1}{b_f}+\frac{1}{b_i}\right)\det D_{\Gamma^{m,p-1}}-e^{-\omega b_i x_f} e^{-\omega b_f x_f}\left(\frac{1}{b_f}-\frac{1}{b_i}\right)\det \tilde D_{\Gamma^{m,p-1}}  $$
$$= e^{\omega b_i x_f} e^{-\omega b_f x_f}\left(\frac{1}{b_f}+\frac{1}{b_i}\right)\det D_{\Gamma^{m,p-1}}\left(1-e^{-2\omega b_i x_f}\frac{\frac{1}{b_f}-\frac{1}{b_i}}{\frac{1}{b_f}+\frac{1}{b_i}} \frac{\det\tilde D_{\Gamma^{m,p-1}} }{\det D_{\Gamma^{m,p-1}} }\right).$$
Therefore $P(1,p-1)$ implies the first part of property $P(1,p)$. In a similar way we compute
$$\det \tilde D_{\Gamma^{m,p}}= e^{\omega b_i x_f} e^{\omega b_f x_f}\left(-\frac{1}{b_f}+\frac{1}{b_i}\right)\det D_{\Gamma^{m,p-1}}-e^{-\omega b_i x_f} e^{\omega b_f x_f}\left(-\frac{1}{b_f}-\frac{1}{b_i}\right)\det \tilde D_{\Gamma^{m,p-1}},$$
so
$$ \frac{\det\tilde D_{\Gamma^{m,p}} }{\det D_{\Gamma^{m,p}} }=
e^{2\omega b_fx_f}
\frac{\frac{-\frac{1}{b_f}+\frac{1}{b_i}}{\frac{1}{b_f}+\frac{1}{b_i}}+e^{-2\omega b_ix_f} \frac{\det\tilde D_{\Gamma^{m,p-1}} }{\det D_{\Gamma^{m,p-1}} }}{1+e^{-2\omega b_ix_f}\frac{-\frac{1}{b_f}+\frac{1}{b_i}}{\frac{1}{b_f}+\frac{1}{b_i}}e^{-2\omega b_ix_f} \frac{\det\tilde D_{\Gamma^{m,p-1}} }{\det D_{\Gamma^{m,p-1}} }},$$
and the second part of $P(1,p)$ follows also from $P(1,p-1)$ by noticing that for $a$ real 
$$\left|\frac{a+z}{1+az}\right|<1\iff 0<(1-a^2)(1-|z|^2).$$
In conclusion $P(1)$ is proved.

\subsection{Proof of $P(n-1)\Rightarrow P(n)$.} 
Let us consider a vertex $v$ from which  there start $m\geq 2$ external infinite edges and one internal edge connecting it to the rest of the graph.
The proof of $P(1,p-1)\Rightarrow P(1,p)$, of eliminating discontinuities on infinite edges also works for trees with $n\geq 1$ vertices. So it is enough to prove $P(n)$ for a graph $\Gamma_n$ 
with no discontinuities on the infinite edges emanating from vertex $v$. 
 We call $\frac{1}{b_j^2},1\leq j\leq m$ the coefficients of the laminar Laplacian on the $m$ edges emerging from $v$. With the notations of the previous subsection and from \S\ref{rec} we have
$$D_{\Gamma_n}=\left(\begin{array}{cccc|ccccccccc}&&&&&& & &&&&&\\&&D_{\Gamma_{n-1}}&&&& & &&&&&\\&&&&-e^{\omega b_i x_i}&& & &&&&&\\&&&&-\frac{-e^{\omega b_i x_i}}{b_i}&& & &&&&&\\ \hline&&&e^{-\omega b_ix_f}&e^{\omega b_i x_f}& -1 & &&&&&\\&&&&&1 & -1 & &&&&&\\  &&&&&& & ...&&&&&\\ &&&&&&&&...&&&&\\&&&&&&&&&...&&\\&&&&&&&&&&&&\\&&&&&&&&&&1&-1&\\&&&&&&&&&&&1&-1\\&&&-\frac{e^{-\omega b_ix_f}}{b_i}&\frac{e^{\omega b_i x_f}}{b_i}&\frac{1}{b_1}&\frac{1}{b_2}&...&...&...&\frac{1}{b_{m-2}}&\frac{1}{b_{m-1}}&\frac{1}{b_m} \end{array}\right).$$
By developing with respect to the first column the $m+1\times m+1$ minors,
$$\det D_{\Gamma_n}=\det D_{\Gamma_{n-1}}(e^{\omega b_ix_f}\det D_{\Gamma^m}+(-1)^{m+2}\frac{e^{\omega b_ix_f}}{b_i}(-1)^m)$$
$$-\det \tilde D_{\Gamma_{n-1}}(e^{-\omega b_ix_f}\det D_{\Gamma^m}-(-1)^{m+2}\frac{e^{-\omega b_ix_f}}{b_i}(-1)^m),$$
so using from the previous subsection the values of $\det D_{\Gamma^m}$ and $\det \tilde D_{\Gamma^m}$, 
$$\det D_{\Gamma_n}(\omega)=e^{\omega b_ix_f}\left(\sum_{j=1}^m\frac{1}{b_j}+\frac{1}{b_i}\right)\det D_{\Gamma_{n-1}}-e^{-\omega b_ix_f}\left(\sum_{j=1}^m\frac{1}{b_j}-\frac{1}{b_i}\right)\det \tilde D_{\Gamma_{n-1}}$$
$$=e^{\omega b_ix_f}\left(\sum_{j=1}^m\frac{1}{b_j}+\frac{1}{b_i}\right)\det D_{\Gamma_{n-1}}\left(1-e^{-2\omega b_ix_f}\frac{\sum_{j=1}^m\frac{1}{b_j}-\frac{1}{b_i}}{\sum_{j=1}^m\frac{1}{b_j}+\frac{1}{b_i}}\frac{\det \tilde D_{\Gamma_{n-1}}}{\det D_{\Gamma_{n-1}}}\right).$$
Therefore, using $P(n-1)$ we get the first part of $P(n)$. In a similar way we get
$$\det \tilde D_{\Gamma_n}(\omega)=e^{\omega b_ix_f}\left(\sum_{j=1}^{m-1}\frac{1}{b_j}-\frac{1}{b_m}+\frac{1}{b_i}\right)\det D_{\Gamma_{n-1}}-e^{-\omega b_ix_f}\left(\sum_{j=1}^{m-1}\frac{1}{b_j}-\frac{1}{b_m}-\frac{1}{b_i}\right)\det \tilde D_{\Gamma_{n-1}},$$
so
$$\frac{\det \tilde D_{\Gamma_n}}{\det D_{\Gamma_n}}=\frac{\frac{\sum_{j=1}^{m-1}\frac{1}{b_j}-\frac{1}{b_m}+\frac{1}{b_i}}{\sum_{j=1}^m\frac{1}{b_j}+\frac{1}{b_i}}-e^{-2\omega b_ix_f}\frac{\sum_{j=1}^{m-1}\frac{1}{b_j}-\frac{1}{b_m}-\frac{1}{b_i}}{\sum_{j=1}^m\frac{1}{b_j}+\frac{1}{b_i}}\frac{\det \tilde D_{\Gamma_{n-1}}}{\det D_{\Gamma_{n-1}}}}{1-e^{-2\omega b_ix_f}\frac{\sum_{j=1}^m\frac{1}{b_j}-\frac{1}{b_i}}{\sum_{j=1}^m\frac{1}{b_j}+\frac{1}{b_i}}\frac{\det \tilde D_{\Gamma_{n-1}}}{\det D_{\Gamma_{n-1}}}}.$$
Since for $a,b,c$ real
$$\left|\frac{\frac{a-b+c}{a+b+c}-\frac{a-b-c}{a+b+c}z}{1-\frac{a+b-c}{a+b+c}z}\right|\leq 1\iff 0<ab(\Im ^2z+(\Re z-1)^2)+bc(1-|z|^2),$$
the second part of $P(n)$ follows, completing the proof in the laminar case.

\section{Open Problems}\label{open}
In this paper we have analyzed the dispersive properties for the linear Schr\"odinger equation on trees. We have assumed that the coupling is given by the classical  Kirchhoff's conditions. However there are other  coupling conditions (see \cite{MR1671833}) which allow to define a ``Laplace" operator on a metric graph. To be more precise, let us consider the operator $H$ that acts on functions on the graph $\Gamma$ as the second derivative $\frac {d^2}{dx^2}$, and its domain consists in all functions $f$ that belong to the Sobolev space
$H^2(e)$ on each edge $e$ of $\Gamma$ and satisfy the following boundary condition at the vertices:
\begin{equation}\label{con.1}
A(v){ \bf f} (v)+B(v){\bf f}'(v)=0 \quad \text{for each vertex} \ v.
\end{equation}
Here ${\bf f}(v)$ and ${\bf f}'(v)$ are correspondingly the vector of values of $f$ at $v$ attained from directions of different edges converging at $v$ and the vector of derivatives at $v$ in the outgoing directions.
For each vertex $v$ of the tree we assume that  matrices $A(v)$ and $B(v)$ are of size $d(v)$ and satisfy the following two conditions\\
\begin{enumerate}
\item the joint matrix $(A(v), B(v))$ has maximal rank, i.e.  $d(v)$,\\
\item $A(v)B(v)^T=B(v)A(v)^T$. 
\end{enumerate} 

Under those assumptions it has been proved in \cite{MR1671833} that the considered operator, denoted by $\Delta(A,B)$, is self-adjoint. The case considered in this paper, the Kirchhoff coupling, corresponds to the matrices
$$A(v)=\left(\begin{array}{cccccc}1 & -1 & 0 & \dots & 0 & 0 \\0 & 1 & -1 & \dots & 0 & 0 \\0 & 0 & 1 & \dots & 0 & 0 \\\vdots & \vdots & \vdots &  & \vdots & \vdots \\0 & 0 & 0 & \vdots & 1 & -1 \\0 & 0 & 0 & \vdots & 0 & 0\end{array}\right),\ 
B(v)=\left(\begin{array}{cccccc}0 & 0 & 0 & \dots & 0 & 0 \\0 & 0 & 0 & \dots & 0 & 0 \\0 & 0 & 0 & \dots & 0 & 0 \\\vdots & \vdots & \vdots &  & \vdots & \vdots \\0 & 0 & 0 & \dots & 0 & 0 \\1 & 1 & 1 & \dots & 1 & 1\end{array}\right).$$
More examples of matrices satisfying the above conditions are given in \cite{MR1671833, MR2459885}.

 The existence of the dispersive properties for the solutions of the Schr\"odinger  on a graph
  under general coupling conditions on the vertices $iu_t+\Delta_\Gamma(A,B)u=0$ is mainly an open problem. The resolvent formula obtained in \cite{MR2459885} and \cite{MR2277618} in terms of the coupling matrices $A$ and $B$ might help to understand the general problem. In the same papers there are also some combinatorial formulations of the resolvent in terms of walks on graphs. Such combinational aspects could clarify if the dispersion is possible only on trees or there are graphs (with some of the edges infinite)  with suitable couplings where the dispersion is still true. 
 
It is expected that other results on the Schr\"odinger equation on $\mathbb R$ are still valid on networks. For instance, the smoothing estimate for the linear equation with constant coefficients is still valid. Although its classical proof on $\mathbb R$ relies on Fourier analysis, one may easily adapt the proof in \cite{BP} which uses only integrations by parts and Besovs spaces that can still be defined on a tree using the heat operator. 
Strichartz estimates has been used previously  to treat controllability issues for the NSE in \cite{MR2514738}. The possible applications of the present results in the control context remains to be analyzed. We mention here some previous works on the controllability/stabilization of the wave equation on networks \cite{MR2169126}, \cite{MR2558320}.
 
 Finally, another problem of interest is the study of the dispersion properties for the  magnetic operators analyzed in \cite{MR1937279}, \cite{MR2007178}. 
The analysis in this case is more difficult since in the presence of an external magnetic field the effect of the topology of the graph becomes more pronounced. In contrast with the analysis done here, in the case of magnetic operators the graphs are viewed as structures in the 
 three dimensional Euclidean space $\R^3$ and the orientation of the edges becomes important.


%
%

\end{document}